\documentclass[12pt]{article}
\usepackage{amsmath,amssymb,amsthm,natbib,xcolor,graphicx}
\usepackage[utf8]{inputenc}
\usepackage[english]{babel}
\usepackage[a4paper,portrait, margin=25mm]{geometry}
\usepackage{hyperref}

\def\P{\mathbb{P}}
\def\R{\mathbb{R}}
\def\Z{\mathbb{Z}}
\def\E{\mathbb{E}}

\def\ds{\displaystyle}
\def\l{\lambda}
\newtheorem{thm}{Theorem}[section]
\newtheorem{rema}[thm]{Remark}
\newtheorem{cor}[thm]{Corollary}
\newtheorem{lemma}[thm]{Lemma}

\newcommand{\stir}[1]{\sqrt{2\pi{#1}}\left(\frac{#1}{e}\right)^{#1}}
\newcommand{\stirp}[1]{\sqrt{2\pi{(#1)}}\left(\frac{#1}{e}\right)^{#1}}

\author{Stanislav Volkov\thanks{Lund University, Centre for Mathematical Sciences,
Lund, SE-22100, Sweden} \and Magnus Wiktorsson\footnotemark[2]}
\title{First to reach $n$ game\thanks{This version corrects a few minor typos which were in the published version.}}
\begin{document}
\maketitle

\begin{abstract}
We consider a game with two players, consisting of a number of rounds, where the player who is first to win \(n\) rounds becomes the overall winner. Who wins each individual round is governed by a certain urn having two types of balls (type 1 and type 2). At each round, we randomly pick a ball from the urn, and its type determines which of the two players wins. We study the game under three regimes: in the first regime, the round outcomes are independent with fixed probabilities; in the second regime, probabilities are governed by a P\'olya urn; finally, in the third regime, balls are sampled without replacement. We obtain the properties of the random variables equal to the properly defined overall net profits of the players, and the results are drastically different in all three regimes.
\\[.3cm]\textbf{Keywords:} Gambler's ruin problem, Rubin's construction, P\'olya's urn
\\[.3cm]\textbf{AMS 2010 subject classifications: 60C05, 60J10, 91A05}
\end{abstract}

\section{Introduction}
Several real-life games and sports use a format that can be viewed as the ``first to reach $n$ wins'' rule. Under this rule, two players play a number of rounds, such that there is a clear winner in each round.  The match continues until one player wins exactly $n$ rounds (or sets), at which point the game ends immediately. The person who won $n$ rounds is declared the overall winner and gets a certain reward. This structure is common in fighting games like {\em Street Fighter} or {\em Mortal Kombat}, racket sports such as tennis, table tennis, and badminton, as well as in some chess championship formats where the first player to win a set number of games is declared the winner. Unlike games with a fixed number of rounds, these formats feature a variable-length match with a clear win condition based on reaching a target number of victories, making them natural real-world analogues of the probabilistic game model described below, where the game continues until one side reaches $n$ wins.

Somewhat relevant models were considered in~\cite{Antal} and~\cite{Svante}. In~\cite{Antal}, the authors studied when the number of balls in the P\'olya's urn becomes equal, and~\cite{Svante} investigated the question of ``eternal lead'' of one of the players. A very recent paper by~\cite{Svante2} (published after we submitted ours) addresses some other aspects of a similar model.

Our models also naturally fall in the general category of time-inhomogeneous Markov chains, see e.g.~\cite{EV}.

The formal description of our initial model is as follows. Given a positive integer~$n$, two players engage in the following game, consisting of at most $2n-1$ rounds.  In round $i=1,2,3,\dots$, Player 1 wins with probability $p_i\in(0,1)$, and Player 2 wins with probability $q_i=1-p_i$. The game continues until the first time that one of the players has won exactly~$n$ rounds. The first player to reach $n$ wins is declared the overall winner and receives a reward (from the other player) equal to the difference between their number of wins (which is $n$ for the overall winner) and the number of rounds won by the losing opponent, denoted by $W_{n,p}\in\{0,1,\dots, n-1\}$. Let $Z_{n,p}\in[-n,n]$ be the (random) net profit of Player 1 as a function of $n$ and $p$. Note that 
$$
Z_{n,p}=(n-W_{n,p}) \mathbf{1}_{\{\text{Player 1 won}\}}+(W_{n,p}-n) \mathbf{1}_{\{\text{Player 2 won}\}}.
$$
The goal of this paper is to study the properties of $W_{n,p}$ and $Z_{n,p}$.

We will consider {\bf three} versions of the above setup. In the first model (``Constant model'') $p_i\equiv p$ is the same for all rounds. The second model (``P\'olya model'') assumes that the probabilities of winning are governed by a P\'olya urn and thus the probability of winning a round for a player increases with the number of wins so far by that player.  The third model (``Anti-OK Corral model'') assumes precisely the opposite. Detailed descriptions of the models will be given further in the paper.

The first model, with constant $p_i$, provided the original motivation for this paper and contains our main results.

\section{Constant model}
In this version of the model, we assume that $p_i\equiv p$ and $q_i\equiv q=1-p$ for all $i$. Note that this model can be viewed as an ``asymptotic'' urn model described in the abstract, with very large initial numbers of balls, $N_1$ and $N_2$, such that $N_1/N_2\approx p/q$. Formally, if $N_1=\lfloor pt\rfloor$ and $N_2=\lfloor qt\rfloor$, as $t\to+\infty$, we obtain the current setup.

A similar model was also studied in~\cite{Bratislava}. In that paper, the game setup remained the same, but one player had the option to stop the game before playing all rounds. The authors were interested in finding the optimal strategy to maximise the expected profit.

\subsection{Martingale approach}
Before stating the exact results, we want to provide some intuition behind them. The intuition is based on constructing a certain martingale, which is a quite usual technique in such models (see e.g.~\cite{King}).

Let $X_k$ be the current difference between wins of the first and the second player after round~$k$, with $X_0=0$. Then $M_k=X_k-\mu k$, where $\mu=p-q$,  is a martingale. Letting $\tau$ be the stopping time when $|X_k|$ hits the ``line'' $2n-k$ for the first time. 
This can at the latest happen when $k=2n-1$. Since $\tau\le 2n-1$ using the Optional Stopping Theorem we get
$$
0=M_0=\E M_\tau=\E X_\tau-\mu \E\tau
$$
where $\tau$ is the index of the last round. On the other hand,
$|X_\tau|=n-(\tau-n)$, implying
$$
\E |X_\tau|=2n-\E\tau, \quad \E |X_\tau|\ge |\E X_\tau|=|\mu| \E\tau
$$
yielding
\begin{align}\label{eqmart}
\E\tau \le \frac{2n}{1+|\mu|}
\quad\Longrightarrow\quad
|\E X_\tau| \le \frac{2n|\mu |}{1+|\mu|}=\frac{n|\mu|}{\max(p,q)}.    
\end{align}
When $p>q$, giving $|\mu|=\mu$, and $n$ is large, it seems quite unlikely that the second player might reach $n$ wins before the first one does. As a result, with probability close to one $X_\tau$ will be positive, and hence $\E |X_\tau|\approx \E X_\tau$, so that we can expect that the inequalities in~\eqref{eqmart} become asymptotically equalities.

Another useful martingale is available in the case $p=q=1/2$. In this case, $\tilde M_k=X_k^2-k$ is also a martingale, hence
$$
0=\tilde M_0=\E \tilde M_\tau=\E X_\tau^2-\E\tau
$$
On the other hand, $|X_\tau|=2n-\tau$ as before, hence, by Jensen's inequality
$$
0=4n^2-4n\E\tau +\E(\tau^2)-\E\tau\ge  4n^2-(4n+1)\E\tau +(\E \tau)^2
$$
yielding
\begin{align}\label{Mart2}
\E\tau\ge 2n-\frac{\sqrt{8n+1}-1}{2} \quad\Longrightarrow\quad
\E|X_\tau|\le \frac{\sqrt{8n+1}-1}{2}.    
\end{align}

\subsection{Main results}
Now, note that by symmetry it suffices to consider only the case $p\ge 1/2$, so we will assume this in the rest of the section. Denote by $\mu=p-q\ge 0$ the expected gain of one round.
The following statement gives an exact explicit formula for $E_{n,p}$.
\begin{thm}\label{T1}
 Let  $E_{n,p}=\E Z_{n,p}$ and $z=pq$. Then 
\begin{align}
   E_{n,p} =n\mu \,\sum_{j=0}^{n-1} C_j z^j
\end{align}
where $C_j=\frac{(2j)!}{j!(j+1)!}$ are {\em Catalan numbers}.
\end{thm}
\begin{rema}
 Catalan numbers have the generating function    
$$
c(z)=\sum_{j=0}^\infty C_j z^j=\frac{1-\sqrt{1-4z}}{2z}
=1+z+2z^2+5z^3+14z^4+42z^5+132z^6+\dots.
$$
hence $E_{n,p}/(n\mu)$ are the first $n$ terms in the Taylor expansion of $c(z)$.
\end{rema}
\begin{cor}
We have 
$$
\lim_{n\to \infty} \lim_{p\downarrow \frac12}  \frac{E_{n,p}}{n\mu}=2.
$$
\end{cor}

We also want to study the asymptotic behaviour of $W_{n,p}$ and $Z_{n,p}$ as $n\to\infty$. 
\begin{thm}\label{T2}
We have
\begin{itemize}
\item[(a)] For  all $0\le k<n$
$$
\P(W_{n,p}=k)=\binom{n+k-1}{n-1}q^k p^n+\alpha_{n,p,k}
$$
where $|\alpha_{n,p,k}|\le (4pq)^n$.
\item[(b)] As $n\to\infty$,
$$
\frac{pZ_{n,p}-\mu n}{\sqrt{nq}}\to \mathcal{N}(0,1) 
\text{ in distribution.}
$$
\end{itemize}
\end{thm}
\begin{rema}
While  part (a) of Theorem~\ref{T2} formally applies to the case $p=1/2$, it is meaningful only  for $p\in(1/2,1)$.
\end{rema}
\begin{rema}\label{rem3}
Part (b) of Theorem~\ref{T2} immediately implies that $\frac{Z_{n,p}}n\to 1-\frac qp$ in probability.
\end{rema}

\begin{rema}\label{rem5}
Part (b) of Theorem~\ref{T2}, together with the tail estimate in Lemma~\ref{lemMarkov}, implies  that in case $p=q=1/2$ the expected profit of the winning player is asymptotically $2\sqrt{n/\pi}\approx 1.13 \sqrt{n}$ (compare this with~\eqref{Mart2}).
\end{rema}

\section{P\'olya's model}
In this version of the game, we assume that we have a classical P\'olya urn initially having $N_i\ge 1$ balls of type $i=1,2$. At each round, we take a ball from the urn at random, and then return it to the urn with one more ball of the same type. The type of ball drawn determines which of the two players wins the round.

\begin{thm}\label{thPolya}
In the P\'olya model, for fixed \(N_1,N_2\ge1\), as the target level \(n\to\infty\),
$$
\P(\text{Player 1 wins overall})\to
\frac {\Gamma(N_1+N_2)}{\Gamma(N_1)\Gamma(N_2)}\,
\int_0^1 \frac {(1-x)^{N_2-1}}{(2-x)^{N_1+N_2}}\,dx.
$$
\end{thm}
\begin{proof}

Recall that the P\'olya's urn model can be realised as follows.  Let $\xi\sim \mathrm{Beta}(N_1,N_2)$ and then, given $\xi$, at each round draw a ball of type 1 (2 respectively) with probability $\xi$ ($1-\xi$ respectively), independently of the past. Then at each round the distribution of balls will be exactly as in the P\'olya's urn. Hence, by Theorem~\ref{T2}, part (b), and by symmetry when \(\xi<1/2\), taking into account that we are no longer guaranteed  the first player winning,
$$
\P\left(\frac{\max(\xi,1-\xi)\, \tilde Z_{n} -n(2\xi-1)   }{\sqrt{n\min(\xi,1-\xi)}}\le a \mid \xi \right)\to \Phi(a)
$$
where $\tilde Z_{n}$ is the overall signed win of the first player.
In particular, this implies that conditionally on~$\xi$,
$$
\frac{\tilde Z_n}{n}\to \frac{2\xi-1}{\max(\xi,1-\xi)}=
\begin{cases}
    \frac{\xi}{1-\xi}-1, &\text{if }\xi<1/2;\\
    1-\frac{1-\xi}{\xi}, &\text{if }\xi\ge 1/2
\end{cases}
$$
(see Remark~\ref{rem3}), hence $\tilde Z_n/n$ converges (unconditionally) in distribution 
to a random variable $\zeta$ with the density
$$
f_\zeta(x)=
\frac{\Gamma(N_1+N_2)}{\Gamma(N_1)\Gamma(N_2)\, (2-|x|)^{N_1+N_2}}\cdot \begin{cases}
(1-|x|)^{N_2-1},&\text{ if }x\in[0,1];\\
(1-|x|)^{N_1-1},&\text{ if }x\in[-1,0].
\end{cases}
$$
Since the limiting random variable \(\zeta\) has no atom at \(0\), the convergence
in distribution implies
\[
\P(\tilde Z_n>0)\to \P(\zeta>0)=\int_0^1 f_\zeta(x)\,dx,
\]
yielding the statement of the Theorem.
\end{proof}

\begin{rema}
For fixed \(r\ge 2\), suppose that \(N_1=N_2=r\), and let the target level
\(n\to\infty\). Since \(|Z_n|/n\le 1\), the convergence in distribution above
also gives convergence of the normalized expected profit of the winning player.
Thus the limiting expected profit of the winning player, divided by \(n\), is
\[
\int_{-1}^1 |x|f_{\zeta,r}(x)\,dx
=
\frac{2(2r-1)!}{(r-1)!^2}
\int_0^1
x\frac{(1-x)^{r-1}}{(2-x)^{2r}}\,dx =
\frac{2}{r-1}
\left(
\frac{(2r-1)!}{(r-1)!^2 2^{2r-1}}-\frac12
\right).
\]
As \(r\to\infty\), this becomes $\frac{2}{\sqrt{\pi r}}+O(r^{-1}).$
Therefore, for fixed \(r\) and target level \(n\to\infty\), the asymptotic expected profit is $\frac{2n}{\sqrt{\pi r}}+O(n/r)$ (compare this with Remark~\ref{rem5}.)
\end{rema}

\section{Anti-OK Corral model}
This version of the model is also governed by an urn. Suppose that there are initially~$n$ balls of each type ($1,2$) in the urn. In each round, we take a ball at random {\em without replacement}, and the type of ball determines who wins the round. The person who first gets $n$ balls of its type (that is, has won $n$ rounds) wins the game. 

In contrast to the original OK Corral model studied in~\cite{King,KV,WM}, where a team with more remaining ``fighters'' gains an advantage, here the player with fewer remaining balls is more likely to be the winner. For this reason we refer to this model as an ``anti-OK Corral'' model.

\begin{thm}\label{thAntiOK}
In the anti-OK Corral model, for each fixed $k\ge 1$, as $n\to\infty$, the probability that a specified player, say Player 1,  wins overall with the net profit of $k$, and the other player wins $n-k\in\{0,1,\dots,n-1\}$ converges to $\frac1{2^{k+1}}$. As a result, the limiting distribution of the net profit of a player is an equal mixture of two Geometric ($1/2$) distributions supported on positive integers.
\end{thm}

Note that it is possible to construct a Markov process that resembles the random walk for the constant case. The issue is that the transition probabilities become both state and time-dependent.  Let \(D_k\) be the difference between the numbers of remaining balls of type one and type two after \(k\) draws. Thus $D_0=0$. Now we have
\begin{eqnarray*} P(D_{k+1}=1+D_k|D_k)&=&\frac{1}{2}\left(1-\frac{D_k}{2n-k}\right),\\
P(D_{k+1}=-1+D_k|D_k)&=&\frac{1}{2}\left(1+\frac{D_k}{2n-k}\right).
\end{eqnarray*}
This is, in fact, equivalent to a Markov representation of a simple random walk bridge on the time interval $0$ to $2n$. We now run this process up to the stopping time $\tau^{(n)}$ when $|D_k|$ hits the ``line'' $2n-k$ for the first time. Since the total number of balls remaining at time $k$ is $2n-k$, it is evident that all remaining balls must be of the same type at the first time $|D_k|=2n-k$ and thus all of the $n$ balls of one of the two types have been taken. If $D_{\tau^{(n)}}=2n-\tau^{(n)}$ then all the remaining balls are of type one and if $D_{\tau^{(n)}}=-(2n-\tau^{(n)})$ then all remaining balls are of type two.

\begin{proof}[Proof of Theorem~\ref{thAntiOK}]
We can solve the problem by direct combinatorial calculations. The probability that the first player wins overall and the other player wins $n-k\in\{0,1,\dots,n-1\}$ rounds is
\begin{equation}
\frac{\ds\frac{n!}{(n-1)!}\frac{n!}{(n-k)!}}{\ds \frac{(2n)!}{((n-1)+(n-k))!}}
=\frac{\ds n (2n-1-k)! n!}{\ds (2n)!(n-k)!} .  \label{Eq:urn}
\end{equation}
In particular, if $k=1$, \eqref{Eq:urn} becomes 
$$
\frac{n(2n-2)!n!}{(2n)!(n-1)!}=\frac{1}{2}\frac{n}{2n-1}=\frac14+o(1).
$$
By Stirling's formula, the expression in \eqref{Eq:urn} is asymptotically equal to
 \begin{align*}
 &\frac{n\stir{n}\stirp{2n-1-k}}{\stirp{n-k}\stir{2n}}\\
 &=\sqrt{\frac{n(2n-1-k)}{(n-k)2n}}\frac{e\left(1-\frac{k+1}{2n}\right)^{2n}}{\left(1-\frac{k}{n}\right)^{n}}2^{-k-1}\left(\frac{2n-2k}{2n-1-k}\right)^{k}\frac{2n}{2n-1-k}.
 \end{align*}
 As $n$ tends to infinity, this converges to  
 $$
 e\cdot \frac{e^{-k-1}}{e^{-k}}2^{-1-k}=2^{-k-1}
 $$
 as required.
 \end{proof}

\section{Remaining proofs}

\begin{proof}[Proof of Theorem~\ref{T1}] 
We assume \(p\ne q\); the case \(p=q=1/2\) then follows by continuity
(or directly, since both sides are zero). An easy calculation shows that
$$
E_{n,p} = \sum_{j = 0}^{n-1} \P(\text{Player 1 wins, }W_{n,p}=j) \cdot (n - j)
- \sum_{i = 0}^{n-1} \P(\text{Player 2 wins, }W_{n,p}=i) \cdot (n - i).
$$
Before the last round, the game has lasted \(n+k-1\) rounds, with the eventual winner having \(n-1\) wins and the opponent having \(k\) wins. The last round is then won by the eventual winner. The probability that the winning player whose probability of winning is~$\theta$ reaches $n$ wins while the other player has exactly $k$ wins equals
$$
h(\theta,k) = \binom{n + k - 1}{k} \cdot \theta^n \cdot (1 - \theta)^k,
$$
where $\theta=p$ or $\theta=1-p$. Then the expected net profit is:
\begin{align*}
E_{n,p} &= \sum_{k=0}^{n-1}(n - k)h(p,k)   
 - \sum_{k=0}^{n-1}(n - k)h(q,k)  
= \sum_{k=0}^{n-1}  \binom{n + k - 1}{k} \left( p^n q^k-q^n p^k   \right)(n - k)
\\ &=
\sum_{k=0}^{n-1}  \binom{n + k - 1}{k} (pq)^{k}\left( p^{n-k}-q^{n-k}   \right)(n - k)
\\ & =
(p-q)\sum_{k=0}^{n-1}  \binom{n + k - 1}{k} (n - k)(pq)^{k}A_{n-k}
\end{align*}
where $A_m=\frac{p^m-q^m}{p-q}$. Then for $m=1,2,\dots$, noting that $p+q=1$,
$$
A_m=\sum_{i=0}^{\lfloor\frac{m-1}2\rfloor}\binom{m-1-i}{i}(-z)^i,\quad z=pq
$$
(this can be shown by induction as $A_{m+2}=A_{m+1}-zA_m$, and $A_1=A_2=1$).
Hence, with $z=pq$, using the change of order of summation, we have
\begin{align*}
\frac{E_{n,p}}{p-q}
&=\sum_{k=0}^{n-1}  \binom{n+k-1}{k} (n - k)z^k
\sum_{i=0}^{\lfloor\frac{n-k-1}2\rfloor}\binom{n-k-1-i}{i}(-z)^i
\\ &=
\sum_{k=0}^{n-1}  \binom{n-1+k}k(n- k)
\sum_{j=k}^{\lfloor\frac{n+k-1}2\rfloor}z^j\binom{n-1-j}{j-k}(-1)^{j-k}
\\ &=
\sum_{j=0}^{n-1} (-1)^j z^j 
\sum_{k=0}^{j}\binom{n-1+k}k \binom{n-1-j}{j-k}(-1)^k(n-k)
\\ &=
\sum_{j=0}^{n-1} (-1)^j z^j \times \frac{(-1)^j\ n}{j+1}\binom{2j}{j}
=n \sum_{j=0}^{n-1} \frac{ z^j}{j+1}\binom{2j}{j}
\end{align*}
where the penultimate equality is due to Lemma~\ref{lemmain}.
\end{proof}

\begin{lemma}\label{lemmain}
 \begin{align}\label{main}
 \sum_{k=0}^{j}\binom{n-1+k}k \binom{n-1-j}{j-k}(-1)^{k}
(n-k)=\frac{(-1)^j\ n}{j+1}\binom{2j}{j}=(-1)^j\ n\ C_j
 \end{align}
\end{lemma}
\begin{proof}
First, we will use the following (probably, known) identity. 
Assume $a\ge b+m$, $m\ge 0$. Then
\begin{align}\label{ident}
\sum_{k=0}^m (-1)^k\binom{a+k}{a}  \binom{b}{m-k}
=(-1)^m\binom{a-b+m}{m}.   
\end{align}
Indeed, the LHS of \eqref{ident} is the coefficient on $x^{a+m}$ in
\begin{align*}
&\sum_{u=0}^{b} \sum_{\ell=a}^\infty \binom{\ell }{a} \binom{b}{u}  (-1)^{\ell-a} x^{u+\ell}
=\left[\sum_{u=0}^{b} \binom{b}{u} x^u\right]\times
\left[\sum_{\ell=a}^\infty \binom{\ell }{a} (-1)^{\ell-a} x^{\ell}\right]    
\\ & =(1+x)^{b} \times \frac{x^a}{(1+x)^{a+1}}=x^a (1+x)^{-(a-b+1)}
=x^a\sum_{m=0}^\infty (-1)^m \binom{a-b+m}{m} x^m.
\end{align*}

Next,
\begin{align*}
\sum_{k=0}^j   \binom{n-1+k}k  \binom{n-1-j}{j-k} (-1)^k (-k)
&=n\sum_{k=0}^{j-1} (-1)^k  \binom{n+k}{k}  \binom{n-1-j}{j-1-k}.
\end{align*}
We will now use twice the identity proven above.
Setting $a=n-1$, $m=j$  ($a=n$, $m=j-1$ respectively) and $b=n-1-j$  in~\eqref{ident} gives
\begin{align*}
\sum_{k=0}^j (-1)^k\binom{n-1+k}{k}  \binom{n-1-j}{j-k}&=(-1)^j\binom{2j}{j};\\
\sum_{k=0}^{j-1} (-1)^k\binom{n+k}{k}  \binom{n-1-j}{j-1-k}&=(-1)^{j-1}\binom{2j}{j-1}
\end{align*}
which yields~\eqref{main} as $C_j=\binom{2j}{j}-\binom{2j}{j-1}$.
\end{proof}

\vspace{1cm}
Before the proof of  Theorem~\ref{T2}, we state the following elementary lemma.

\begin{lemma}\label{lemMarkov}
\begin{itemize}
    \item[(a)] 
Suppose that $\tau_X$ and $\tau_Y$ are $\Gamma(n,1)$ and $\Gamma(n,\l)$ respectively random variables, which are independent and $0<\l<1$. Then
    $$
     \P(\tau_X - \tau_Y\ge 0)\le \left(\frac{4\l}{(\l+1)^2}\right)^n\to 0\text{ as }n\to\infty;
    $$
\item[(b)] Let $\eta$ be a negative binomial random variable with parameters $m$ and $1/2$. Then
$$
\P(|\eta-m|\ge a)\le 2 \exp\left(-\frac{a^2}{4(a+m)}\right)\quad\text{for }a\ge 0.
$$    
\end{itemize}
\end{lemma}
\begin{proof}
(a) We can write $\tau_X=\xi_1+\dots+\xi_n$, $\tau_Y=\eta_1+\dots+\eta_n$, 
where $\xi_i$s ($\eta_i$s resp.) are $\exp(1)$ i.i.d.\ random variables with rate $1$ ($\l$ resp.) and $\xi$s and $\eta_i$s are independent. Hence, by Markov's inequality,
\begin{align}\label{Cramer}
\begin{split}
\P\left(\tau_X- \tau_Y\ge 0 \right)&=\P\left(\sum_{i=1}^n (\xi_i-\eta_i) \ge 0\right)
\le \inf_{0<t<1}\P\left(e^{t\sum_{i=1}^n (\xi_i-\eta_i)}\ge 1\right)
\\ & \le \inf_{0<t<1} \E\left(e^{t\sum_{i=1}^n (\xi_i-\eta_i)}\right)
= \inf_{0<t<1} \left(\frac{\l}{(1-t)(\l+t)}\right)^n.
\end{split}
\end{align}
Noting the minimum is achieved at  $t=\frac{1-\l}{2}\in(0,1)$ finishes the proof.
\vspace{5mm}

(b) First, note that we can write $\eta$ as $\zeta_1+\zeta_2+\dots+\zeta_m$ where $\zeta_i$s are i.i.d.\ Geometric($1/2$) random variables supported on $\{0,1,2,\dots\}$. Hence, by Markov's inequality, for all $t>0$ (where the expressions below are defined) we have
\begin{align*}
\P(\eta\ge m+a)&\le e^{-t(m+a)}\E e^{t\eta}=e^{-I(t)}
\quad \text{where }I(t)=t(m+a)+m\log(2-e^t);\\
\P(\eta\le m-a)&\le e^{t(m-a)}\E e^{-t\eta}=e^{-\tilde I(t)}
\quad \text{where }\tilde I(t)=-t(m-a)+m\log\left(2-e^{-t}\right)    
\end{align*}
(note that we need to estimate the second inequality only for $a\le m$).
Maximizing $I(t)$ and $\tilde I(t)$ over $t>0$ gives 
$$
\sup_{t>0} I(t)=m\, f\left(\frac am\right),\qquad 
\sup_{t>0} \tilde I(t)=m \,\tilde f\left(\frac am\right)
$$
where
\begin{align*}
f(y)&=y\ln\left(1+\frac{y}{2+y}\right)+\ln\left(1-\frac{y^2}{(2+y)^2}\right);\\
\tilde f(y)&=(1-y)\ln(1-y) +(2-y) (\ln(2)-\ln(2-y)).
\end{align*}
Let $g(y)=\frac{y^2}{4(y+1)}$. 
Since $f(0)=f'(0)=g(0)=g'(0)=0$ and 
$$
f''(y)=\frac1{(1+y)(2+y)}\ge \frac1{2(1+y)^3}=g''(y)
$$
we have $f(y)\ge g(y)$ and thus $\sup_{t>0} I(t)\ge m g(a/m)=\frac{a^2}{4(a+m)}$.

Similarly, $\tilde f(0)=\tilde f'(0)=0$ and $\tilde f''(y)=\frac{1}{(y-1)(y-2)}\ge 1/2$ for $y\in[0,1]$, so  $\tilde f(y)\ge y^2/4$ and thus $\sup_{t>0}\tilde I(t)\ge \frac{a^2}{4m}\ge \frac{a^2}{4(m+a)}$. From this the statement follows.
\end{proof}

\begin{proof}[Proof of Theorem~\ref{T2}] 
We separately consider the cases when $p=1/2$ and when $p>1/2$. We will use the Poisson process representation, similar to Rubin's construction in~\cite{BD} or the one in~\cite{KV}.

Let $\l=q/p\le 1$. Consider two independent Poisson processes, $X(t)$ with rate $1$, and $Y(t)$ with rate $\l$. Let 
$$
\tau_X=\tau_X^{(n)}=\min\{t:\ X_t=n\},\quad 
\tau_Y=\tau_Y^{(n)}=\min\{t:\ Y_t=n\}, \quad \tau=\tau^{(n)}=\min(\tau_X,\tau_Y).
$$
Then $Z=X(\tau)-Y(\tau)$ has the same distribution as our quantity of interest $Z_{n,p}$. Indeed, the embedded Markov chain for the process $(X(t),Y(t))$, $t\ge 0$, has exactly the same transition probabilities as in the game. Also note that $\P(\tau_X=\tau_Y)=0$, and thus we can write
\begin{align}\label{eqPois}
 Z_{n,p}=[n-Y(\tau_X)]\mathbf{1}_{\tau_X<\tau_Y}+[X(\tau_Y)-n]\mathbf{1}_{\tau_X\ge \tau_Y}.
\end{align}

First, let us deal with the case $p>1/2$, i.e. $\l<1$. In this case
$$
\P\left(\tau_X^{(n)}\ge  \tau_Y^{(n)}\right)\le (4pq)^n \to 0 \text{ as }n\to \infty.
$$
by part (a) of Lemma~\ref{lemMarkov}, since $p>1/2$ implies $4pq=1-(2p-1)^2<1$.
Since
\begin{align*}
\P(W_{n,p}=k)&=\P(Y(\tau_X)=k, \tau_X<\tau_Y)+
\P(X(\tau_Y)=k, \tau_X>\tau_Y)
\\ & =\P(Y(\tau_X)=k)-\P(Y(\tau_X)=k, \tau_X>\tau_Y)
+\P(X(\tau_Y)=k, \tau_X>\tau_Y)
\end{align*}
from~\eqref{Cramer} we obtain
\begin{align}\label{eqZ}
\left|\P(W_{n,p}=k)-\P(Y(\tau_X)=k)\right| 
\le \P(\tau_X>\tau_Y)\le (4pq)^n.   
\end{align}
Now, the only remaining problem is to estimate the distribution of~$Y(\tau_X)$. This can be computed explicitly, since $\tau_X\sim\Gamma(n,1)$:
\begin{align*}
\begin{split}   
 \P(Y(\tau_X)=k)&=\E \left[\P(Y(\tau_X)=k\mid \tau_X)\right]
 =\int_0^\infty \frac{(\l u)^k e^{-\l u}}{k!}\cdot \frac{u^{n-1}e^{-u}}{(n-1)!}du\\ &
=\frac{\l^k}{(\l+1)^{n+k}}\frac{ (n+k-1)!}{k ! (n-1) !}
= \binom{k+n-1}{k}  q^k  p^n,\quad k\in\Z_+,
\end{split}
\end{align*}
which is a negative binomial distribution with mean $\l n$ and variance $\l n/p$ (this can also be obtained directly using the coin tossing model for the negative binomial distribution). Combining this with \eqref{eqZ} finishes the proof of part (a) of the Theorem in case $p>1/2$. Finally, when \(p=1/2\), part (a) is immediate from the definition of \(\alpha_{n,p,k}\), since then~\((4pq)^n=1\).
\vspace{5mm}

Next, continuing with $p>1/2$ case, by the Central Limit Theorem applied to the negative binomial distribution, for any $a,b\in \R$ with $a<b$ we get 
\begin{align}\label{eqXY}
\lim_{n\to\infty}\P\left(a\le \frac{Y(\tau_X)-\l n}{\sqrt{\l n/p}}\le b \right)
= \Phi(b)-\Phi(a)
\end{align}
where $\Phi(\cdot)$ is the cumulative distribution function of the standard normal distribution.
At the same time, as $n\to\infty$,
\begin{align*}
\left|\P\left(\frac{Y(\tau_X)-\l n}{\sqrt{\l n/p}}\in[a,b] \right)
-\P\left(\frac{W_{n,p}-\l n}{\sqrt{\l n/p}}\in[a, b] \right)
\right|
&\le \left[\sqrt{\frac{\l n}p}(b-a)+1\right](4pq)^n\to 0
\end{align*}
by~\eqref{eqZ}. This, together with~\eqref{eqXY}, implies the CLT for $W_{n,p}$ and hence also for $Z_{n,p}$, since
\begin{align*}
\P\left(Z_{n,p}\in[a,b]\right)&=\P\left(n-W_{n,p}\in[a,b],\tau_X<\tau_Y\right)
+\P\left(W_{n,p}-n\in[a,b],\tau_X>\tau_Y\right),
\\
\P\left(n-W_{n,p}\in[a,b]\right)
&=\P\left(n-W_{n,p}\in[a,b],\tau_X<\tau_Y\right)+\P\left(n-W_{n,p}\in[a,b],\tau_X>\tau_Y\right)
\end{align*}
together yield
\begin{equation}\label{eqZZWW}
\left|\P\left(Z_{n,p}\in[a,b]\right)-\P\left(n-W_{n,p}\in[a,b]\right)
\right|\le 2\, \P\left(\tau_X>\tau_Y\right)\le 2\, (4pq)^n\to 0   
\end{equation}
as $n\to\infty$.

\vspace{1cm}

Now, assume $p=1/2$. Here, both Poisson processes are of rate $1$, and in this case, each $\tau_X$ and $\tau_Y$ is equally likely to be the minimum one. From~\eqref{eqPois},
\begin{align*}
Z_{n,p}&=[n-Y(\tau_X)](1-\mathbf{1}_{\tau_X\ge \tau_Y})+[X(\tau_Y)-n]\mathbf{1}_{\tau_X\ge \tau_Y}
\\
&=n-Y(\tau_X)+\left[X(\tau_Y)+Y(\tau_X)-2n\right]\mathbf{1}_{\tau_X\ge \tau_Y}.
\end{align*}
Similarly to the case $p>1/2$, we know that $Y(\tau_X)$ is a negative binomial, and thus by the CLT (see~\eqref{eqXY})
$$
\frac{n-Y(\tau_X)}{\sqrt{2n}}\to\mathcal{N}(0,1)\qquad \text{in distribution}.
$$
We will show now that
\[
\frac{\zeta_n\mathbf 1_{\{\tau_X\ge \tau_Y\}}}{\sqrt n}\to 0
\quad\text{in probability, where}\quad
\zeta_n:=X(\tau_Y)+Y(\tau_X)-2n,
\]
thus yielding the desired result.

\begin{figure}
    \centering
    \includegraphics[width=0.6\linewidth]{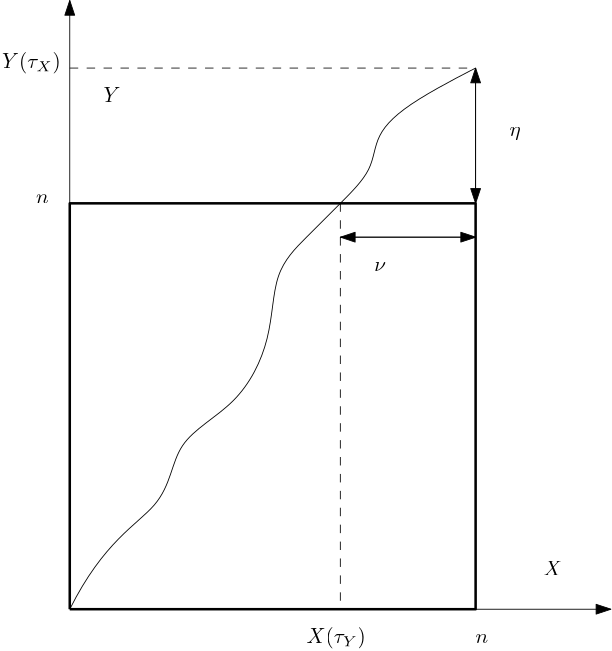}
    \caption{A path of the process in case $\tau_Y<\tau_X$.}
    \label{fig1}
\end{figure}
Indeed, since 
$$
\{\tau_Y<\tau_X\}=\{X(\tau_Y)=n-k \text{ for some }k\ge 1\}
$$
from the strong Markov property it follows that the behaviour of the processes $X(t), Y(t)$ for $t>\tau_Y$ is independent of the past, hence conditionally on $\nu=n-X(\tau_Y)\ge 1$ the distribution of $\eta:=Y(\tau_X)-Y(\tau_Y)\equiv Y(\tau_X)-n$ is negative binomial with parameters~$\nu$ and~$1/2$ (see also Figure~\ref{fig1}). By part (b) of Lemma~\ref{lemMarkov} 
$$
\P(|\eta-\nu|>n^{2/5}\mid \nu )
\le 2 \exp\left(-\frac{n^{4/5}}{4(\nu+n^{2/5})}\right) 
\le 2 \exp\left(-\frac{n^{2/15}}8\right) 
\text{ on }\nu< n^{2/3}.
$$
Now, $X(\tau_Y)$ is also a negative binomial with parameters $n$ and $1/2$, so again by part (b) of Lemma~\ref{lemMarkov}
$$
\P(|\nu|\ge n^{2/3})=\P(|X(\tau_Y)-n|\ge n^{2/3})
\le2 \exp\left(-\frac{ n^{4/3}}{4(n+n^{2/3})}\right)
\le2 \exp\left(-\frac{ n^{1/3}}{8}\right).
$$
Hence
\begin{align*}
\P\left(|\eta-\nu|\ge n^{2/5},\,\tau_Y<\tau_X\right)
&\le
\P\left(|\eta-\nu|\ge n^{2/5},\,1\le \nu<n^{2/3}\right)
+\P\left(\nu\ge n^{2/3},\,\tau_Y<\tau_X\right)
\\
&\le
2\exp\left(-\frac{n^{2/15}}8\right)
+
2\exp\left(-\frac{n^{1/3}}8\right)
\to0.
\end{align*}
On the event \(\tau_Y<\tau_X\), we have
\[
\eta-\nu=(Y(\tau_X)-n)-(n-X(\tau_Y))
=Y(\tau_X)+X(\tau_Y)-2n=\zeta_n.
\]
Since \(n^{2/5}=o(\sqrt n)\), this proves
\[
\frac{\zeta_n\mathbf 1_{\{\tau_X\ge \tau_Y\}}}{\sqrt n}\to0
\]
in probability, as required.
\end{proof}


\begin{thebibliography}{99}
\bibitem[Addona et~al.(2011)]{Bratislava}
Addona, V.,  Wilf, H., and Wagon, S. (2011).
How to lose as little as possible.
\textit{Ars Mathematica Contemporanea}, {\bf 4:1} 29--62.

\bibitem[Antal, Ben-Naim and Krapivsky(2010)]{Antal}
Antal, T., Ben-Naim, E., Krapivsky, P. L. (2010).
First-passage properties of the P\'olya urn process. 
\textit{J.\ Stat.\ Mech.\ Theory Exp.}, no.~7, P07009, 11~pp.

\bibitem[Davis(1990)]{BD}
Davis, B. (1990). Reinforced random walk.  \textit{Probability Theory and Related Fields}, {\bf 84}, 203--229.

\bibitem[Engl\"ander and Volkov(2025)]{EV}
Engl\"ander, J., and  Volkov, S. (2025).  \textit{Coin Turning, Random Walks and Inhomogeneous Markov Chains}, World Scientific.

\bibitem[Janson(2025)]{Svante}
Janson, S. (2025). A note on P\'olya urns: the winner may lead all the time.
\href{https://arxiv.org/abs/2506.14859}{https://arxiv.org/abs/2506.14859}

\bibitem[Janson, Martinez, Zeilberger(2025)]{Svante2}
Janson, S., Martinez, L., and Zeilberger, D. (2025).
How many coin tosses would you need until you get $n$ Heads or $m$ Tails?
\href{https://arxiv.org/abs/2512.07803}{https://arxiv.org/abs/2512.07803}

 
 
\bibitem[Kingman(1999)]{King}
Kingman, J. F. C. (1999). Martingales in the OK Corral. \textit{Bull.\ London Math.\ Soc.}, {\bf 31}, 601--606.

\bibitem[Kingman and Volkov(2003)]{KV}
Kingman, J.F.C., Volkov, S. (2003).
Solution to the OK Corral Model via Decoupling of Friedman's Urn.
\textit{Journal of Theoretical Probability}, {\bf 16} 267--276.

\bibitem[Williams and McIlroy(1998)]{WM}
Williams, D., and McIlroy, P. (1998). The OK Corral and the power of the law (a curious Poisson kernel formula for a parabolic equation). \textit{Bull.\ London Math.\ Soc.}, {\bf 30}, 166--170.
\end{thebibliography}
\end{document}